\documentclass[12pt,scheme=plain,fontset=none]{ctexart}
\usepackage[T1]{fontenc}
\usepackage[utf8]{inputenc}
\usepackage{lmodern}

\usepackage[a4paper,margin=1in]{geometry}
\usepackage{amsmath,amssymb,amsthm,mathtools}
\usepackage{enumitem}
\usepackage{microtype}
\usepackage{xcolor}
\usepackage{hyperref}

\hypersetup{
  colorlinks=true,
  linkcolor=blue,
  citecolor=blue,
  urlcolor=blue
}

\numberwithin{equation}{section}

\AtBeginDocument{%
  \setlength{\abovedisplayskip}{6.5pt plus 2pt minus 2pt}%
  \setlength{\belowdisplayskip}{6.5pt plus 2pt minus 2pt}%
  \setlength{\abovedisplayshortskip}{4pt plus 2pt minus 2pt}%
  \setlength{\belowdisplayshortskip}{4pt plus 2pt minus 2pt}%
}

\theoremstyle{plain}
\newtheorem{theorem}{Theorem}[section]
\newtheorem{corollary}[theorem]{Corollary}
\newtheorem{conjecture}[theorem]{Conjecture}
\newtheorem{proposition}[theorem]{Proposition}
\newtheorem{lemma}[theorem]{Lemma}
\theoremstyle{definition}
\newtheorem{remark}[theorem]{Remark}
\newtheorem*{remark*}{Remark}

\newcommand{\R}{\mathbb R}
\newcommand{\N}{\mathbb N}
\newcommand{\Z}{\mathbb Z}

\newcommand{\E}{\mathbb E}

\title{A Walsh--Quotient Obstruction for Fourier Frames\\
on Odd Reciprocal-Power Bernoulli Convolutions}
\author{Xiao-Ye Fu, Zi-Jian Song, and Wei-Jie Wang}
\date{}

\begin{document}
\maketitle

\begin{abstract}
We introduce a Walsh--quotient obstruction to study Fourier-frame
existence for symmetric two-branch Bernoulli convolutions
\[
\mu_{\rho,d}
=\ast_{j=1}^{\infty}
\frac12\bigl(\delta_{-d\rho^{j}/2}+\delta_{d\rho^{j}/2}\bigr),
\qquad 0<\rho<1,\quad d>0.
\]
Suppose that $0<\rho<\frac12$ and $\rho^{-m}=B$ for some integer
$m\ge1$ and odd integer $B\ge3$. We prove that
$L^2(\mu_{\rho,d})$ admits no Fourier frame. For $m=1$, our argument
proves the nonexistence of Fourier frames for odd-integer-base
Cantor measures and hence resolves Strichartz's long-standing open
problem for the middle-third Cantor measure. A contemporaneous
independent proof of the case $m=1$ was obtained by Pont, Liehr and
Taylor [arXiv:2607.08656v1]. For $m>1$, our theorem includes the non-integer
reciprocal-power contraction ratios $\rho=B^{-1/m}$, which fall
outside the classical integer-base Cantor-measure setting. Our proof
is self-contained. It uses finite-coordinate Walsh packets to
transform the frame inequalities into incompatible tangent-quotient
estimates, while the identity $\rho^{-m}=B$ supplies the exact
$m$-step scale relation leading to the contradiction.
\end{abstract}



\section{Introduction}

Fourier frames address whether a system of exponentials provides stable, though not necessarily orthogonal, expansions in a Hilbert space.  For singular self-similar measures, this question is substantially finer than spectrality.  The standard middle-third Cantor measure is the canonical test case.  We denote by $\mu_{\mathrm C}$ the self-similar probability measure determined by
\[
  \mu_{\mathrm C}
  =\frac12\,\mu_{\mathrm C}\circ\tau_0^{-1}
   +\frac12\,\mu_{\mathrm C}\circ\tau_2^{-1},
  \qquad
  \tau_0(x)=\frac{x}{3},\quad
  \tau_2(x)=\frac{x+2}{3}.
\]
Equivalently, $\mu_{\mathrm C}$ is the distribution of
\(
  \sum_{j=1}^{\infty}\frac{2\varepsilon_j}{3^j},
\)
where $\varepsilon_1,\varepsilon_2,\ldots$ are independent random variables, each uniformly distributed on $\{0,1\}$.  Jorgensen and Pedersen~\cite{JorgensenPedersen1998} proved that $\mu_{\mathrm C}$ is non-spectral: in fact, $L^2(\mu_{\mathrm C})$ does not contain three mutually orthogonal exponentials. Motivated by this negative resolution, Strichartz~\cite{Strichartz2000} subsequently asked whether the weaker Fourier-frame property might nevertheless hold for $\mu_{\mathrm C}$.  This question remained open for more than two decades and continued to serve as a basic test problem in fractal Fourier analysis.  Related developments on spectra and Fourier frames for Cantor-type and fractal measures can be found in \cite{DutkayHanSun2009,DutkayHanWeber2014,DutkayJorgensen2007,DutkayLaiWang2017,Herr2017,LabaWang2002,Lev2018}.  While the present manuscript was under final revision, Pont, Liehr and Taylor~\cite{PontLiehrTaylor2026} independently proved the non-Fourier-frame result for odd integer-base Cantor measures; the precise relation between the two works is explained in Remark~\ref{rem:relation-plt}.

\vspace{0.2cm}

In this paper, we introduce a new technical obstruction, termed the \emph{Walsh--quotient obstruction}, for the Fourier-frame problem.  The method is developed for the symmetric two-branch Bernoulli convolution
\begin{equation}\label{eq:murhod-intro}
  \mu_{\rho,d}
  =\ast_{j=1}^{\infty}
  \frac12\left(\delta_{-d\rho^j/2}+\delta_{d\rho^j/2}\right),
  \qquad 0<\rho<1,
  \quad d>0.
\end{equation}
Equivalently, $\mu_{\rho,d}$ is the distribution of
\begin{equation}\label{eq:murhod-random-series-intro}
  \sum_{j=1}^{\infty}\frac{d\xi_j\rho^j}{2},
  \qquad
  \xi_j\in\{-1,1\},\quad
  \mathbb P(\xi_j=-1)=\mathbb P(\xi_j=1)=\frac12,
\end{equation}
where $\xi_1,\xi_2,\ldots$ are independent sign variables.

\vspace{0.2cm}

Replacing $\xi_j$ by $2\varepsilon_j-1$ shows that the usual one-sided Cantor-type measure, namely the distribution of $\sum_{j=1}^{\infty}d\varepsilon_j\rho^j$, is the translate of $\mu_{\rho,d}$ by $d\rho/[2(1-\rho)]$.  By the affine invariance property proved in Lemma~\ref{lem:affine}, translation preserves Fourier-frame spectra and their frame bounds, as well as spectrality and the existence of exponential Riesz bases.  Thus the symmetric two-branch Bernoulli convolution and the corresponding Cantor-type measure are equivalent for all questions considered below, and we work with the symmetric normalization throughout the paper.

\vspace{0.2cm}

We use the convolution form \eqref{eq:murhod-intro} and the random-series form \eqref{eq:murhod-random-series-intro} interchangeably.  For an integer base $b\ge2$, write $\mu_{b,d}:=\mu_{1/b,d}$.  With the Fourier-transform convention $\widehat\sigma(\lambda)=\int_{\R} e^{-2\pi i\lambda x}\,d\sigma(x)$, the Fourier transform of $\mu_{\rho,d}$ factors as the infinite product of cosine masks
\[
  \widehat{\mu_{\rho,d}}(\lambda)
  =\prod_{j=1}^{\infty}\cos(\pi d\lambda\rho^j).
\]
For $d=1$, \eqref{eq:murhod-intro} is the standard symmetric normalization of the Bernoulli convolution with contraction ratio $\rho$; the parameter $d>0$ is only a dilation.  Dai~\cite{Dai2012} completely classified spectrality in this family: after transfer to the present normalization, $\mu_{\rho,d}$ is spectral if and only if $\rho=1/(2k)$ for some $k\in\mathbb N$.  Dai, He and Lau~\cite{DaiHeLau2014} later obtained the corresponding classification for $N$-Bernoulli measures.  These results settle the exponential orthonormal-basis problem, but they do not settle the Fourier-frame problem, since nonspectrality alone does not imply the nonexistence of a Fourier frame.

\vspace{0.2cm}

Our main result gives a uniform obstruction whenever an odd integer appears after a finite number of scale steps.
\begin{samepage}
\begin{theorem}\label{thm:level-shift-main}
Let $0<\rho<1/2$ and let $d>0$.  Suppose that there are an integer $m\ge1$ and an odd integer $B\ge3$ such that $\rho^{-m}=B$.  Then $L^2(\mu_{\rho,d})$ admits no Fourier frame. 
\end{theorem}
\end{samepage}

The constraint $0<\rho<1/2$ exactly characterizes the non-overlapping condition of $\mu_{\rho,d}$ that guarantees finite-coordinate Walsh packets defined in \eqref{eq:walsh-packet} are well-defined functions in $L^2(\mu_{\rho,d})$. For $m=1$, Theorem~\ref{thm:level-shift-main} gives the odd reciprocal-integer case.  When $m>1$, it includes genuinely non-integer contraction ratios $\rho=B^{-1/m}$ whenever $B$ is not an $m$th power.  Our proof relies on the exact relation induced by the identity
		$\rho^{-m}=B$ and does not necessitate reduction to integer-base Cantor measures.

\begin{remark}[Relation to Pont-Liehr-Taylor's non-Fourier-frame result for odd integer-bases]
\label{rem:relation-plt}
During the preparation of the present manuscript,  Pont, Liehr and Taylor~\cite{PontLiehrTaylor2026} posted their independent proof that odd integer-base Cantor measures do not admit Fourier frames.  Their result treats the case $\rho=1/b$ with $b$ odd, and therefore includes the middle-third Cantor measure. Our argument does not depend on their theorem; we develop a separate Walsh--quotient obstruction that applies to a wider family of contraction ratios satisfying $\rho^{-m}=B$, where $B\ge3$ is odd and $m\ge1$. When $m=1$, our result recovers their non-Fourier-frame conclusion for odd integer-bases, while the cases $m>1$ include new nonexistence results for non-integer reciprocal-power contraction ratios under the non-overlap condition $0< \rho<1/2$. 
\end{remark}

Combining Theorem~\ref{thm:level-shift-main} with Dai's spectrality theorem \cite{Dai2012} yields a full characterization of Fourier frames for reciprocal-integer contraction ratios. Precisely, for every integer $b\ge2$ and every $d>0$, the space $L^2(\mu_{b,d})$ admits a Fourier frame  if and only if $b$ is even. Spectrality holds for all even bases as a direct consequence of Dai's theorem, while the nonexistence of Fourier frames for odd bases can be obtained by taking $m=1$ and $B=b$ in Theorem~\ref{thm:level-shift-main}.
The standard middle-third Cantor measure is an affine copy of $\mu_{3,2}=\mu_{1/3,2}$, one member of the family $\mu_{b,d}$. Theorem~\ref{thm:level-shift-main} with $m=1$ and $B=3$, together with Lemma~\ref{lem:affine}, yields the following corollary.

\begin{corollary}
	The standard middle-third Cantor measure admits no Fourier frames. 
\end{corollary}

To set our results in perspective, we briefly survey closely related work in the literature.
He, Lai and Lau~\cite{HeLaiLau2013} proved that a measure admitting a Fourier frame must be of pure type, while Dutkay and Lai~\cite{DutkayLai2014} established strong uniformity conditions for frame measures. These results do not exclude the equal-weight Cantor measures considered here. On the other hand, nonspectrality alone does not preclude the existence of Fourier frames: Lai and Wang~\cite{LaiWang2017} constructed a singular fractal measure with only finitely many mutually orthogonal exponentials, yet this measure still admits a Fourier frame. Non-Fourier-frame results in other geometric settings include surface-carried measures~\cite{IosevichLaiLiuWyman2022} and one-dimensional Salem measures~\cite{LiLiu2025}. Separately, Picioroaga and Weber~\cite{PicioroagaWeber2017} obtained Parseval frames of weighted exponentials for the scale-four Cantor measure; these weighted systems are more general than the unweighted Fourier frames considered here.

\vspace{0.2cm}

To this end, we sketch the main mechanism underlying our proof. Within the symbolic representation of $\mu_{\rho,d}$, finite-coordinate Walsh packets are defined by
\[
w_F(\xi)=\prod_{j\in F}\xi_j,
\qquad F\subset\mathbb N\text{ finite}.
\]
Testing a hypothetical Fourier frame against these Walsh packets produces the cosine and sine masks
\[
m_j(\lambda)=\cos(\pi d\lambda\rho^j),
\qquad
n_j(\lambda)=-i\sin(\pi d\lambda\rho^j).
\]
Assume that $\Lambda$ is a Fourier frame spectrum for $\mu_{\rho,d}$, with lower and upper frame bounds $A$ and $C$, respectively.  Its counting measure $\Gamma_\Lambda:=\sum_{\lambda\in\Lambda}\delta_\lambda$ is a frame measure for $\mu_{\rho,d}$ in the terminology of \cite{DutkayHanWeber2014,DutkayLai2014}.  Weighting $\Gamma_\Lambda$ by the squared Fourier transform gives the frame-induced measure
\begin{equation}\label{eq:eta-def}
	\eta
	:=|\widehat{\mu_{\rho,d}}|^2\,\Gamma_\Lambda
	=\sum_{\lambda\in\Lambda}
	|\widehat{\mu_{\rho,d}}(\lambda)|^2\,\delta_\lambda.
\end{equation}
The upper frame inequality applied to the constant function gives $\eta(\R)\le C$, so $\eta$ is finite.
Away from the zero set of $m_j$, their quotient is
\[
	\frac{n_j(\lambda)}{m_j(\lambda)}
	=-i\tan(\pi d\lambda\rho^j)
	=:-iu_j(\lambda).
\]
The zero set of $m_j$ is an $\eta$-null set and therefore it does not affect the relevant integrals.  Testing the frame inequalities on one-coordinate and two-coordinate Walsh packets yields that, for every $s>m$,
\[
\int_{\R}|u_s(\lambda)|^2\,d\eta(\lambda)\ge A,
\qquad
\int_{\R}|u_{s-m}(\lambda)u_s(\lambda)|^2\,d\eta(\lambda)\le C.
\]
The algebraic identity $\rho^{-m}=B$ implies that, off an explicit set of poles, 
\[
u_{s-m}=T_B(u_s),
\qquad
T_B(u)=\tan(B\arctan u).
\]
Since $B$ is an odd integer, large magnitudes of $u_s$ necessarily force comparable large magnitudes of $u_{s-m}$.  Combining this with the product integral bound then implies $\int |u_s|^2\,d\eta\to0$, contradicting the uniform lower bound on the integral.  This incompatibility between the two frame inequalities  constitutes the Walsh--quotient obstruction to the existence of Fourier frames.

\vspace{0.2cm}

The paper is organized as follows.  In Section~\ref{sec:preliminaries}, we introduce and review some foundational background: Fourier frame theory, affine invariance properties, the symbolic dynamical model for $\mu_{\rho,d}$, and the construction of Walsh packets. In Section~\ref{sec:obstruction},  we prove the abstract $L^2$-decay lemma and the tangent dynamics generated by $\rho^{-m}=B$. In Section~\ref{sec:walsh-frame}, we establish the Walsh product formulae and their implications for frame theory, and complete the proof of Theorem~\ref{thm:level-shift-main}. In Section~\ref{sec:scope}, we discuss the scope of the method and record the remaining non-overlapping conjecture.

\section{Preliminaries}\label{sec:preliminaries}

Let $\mu$ be a compactly supported Borel probability measure on $\R$.  We equip $L^2(\mu)$ with the inner product
\[
  \langle f,g\rangle_{L^2(\mu)}
  =\int_{\R} f(x)\overline{g(x)}\,d\mu(x),
  \qquad f,g\in L^2(\mu),
\]
and write $e_\lambda(x)=e^{2\pi i\lambda x}$ for $\lambda\in\R$.  A countable set $\Lambda\subset\R$ is said to be a Fourier frame spectrum for $\mu$ if there exist constants $0<A\le C<\infty$ such that
\begin{equation}\label{eq:frame-def}
  A\|f\|^2_{L^2(\mu)}
  \le
  \sum_{\lambda\in\Lambda}
  |\langle f,e_\lambda\rangle_{L^2(\mu)}|^2
  \le
  C\|f\|^2_{L^2(\mu)},
  \qquad f\in L^2(\mu).
\end{equation}
The two-sided estimate in \eqref{eq:frame-def} is called the frame inequality, with $A$ and $C$ referred to as the lower and upper frame bounds, respectively.  A spectral measure is one for which $L^2(\mu)$ admits an orthonormal basis of exponentials.  Every exponential orthonormal basis is a Fourier frame with bounds $A=C=1$, and, more generally, every exponential Riesz basis is a Fourier frame.  For a comprehensive introduction to frames and Riesz bases, we refer readers to~\cite{Christensen2016}.

\vspace{0.2cm}

We state an elementary invariance result that will be used repeatedly.
\begin{lemma}\label{lem:affine}
Let $T(x)=ax+c$ with $a\ne0$, and let $\widetilde\mu=T_\#\mu$ denote the pushforward measure of $\mu$ under $T$.  Then $L^2(\mu)$ admits a Fourier frame if and only if $L^2(\widetilde\mu)$ does.  This invariance property holds for exponential orthonormal bases and exponential Riesz bases.
\end{lemma}
\begin{proof}
The pushforward identity $\int_{\R} g(y)\,d\widetilde\mu(y)=\int_{\R} g(Tx)\,d\mu(x)$ holds for every integrable function $g$.  Hence the operator $U:L^2(\widetilde\mu)\to L^2(\mu)$ defined by $Ug=g\circ T$ is unitary.  Moreover,
\[
  Ue_\lambda(x)=e_\lambda(Tx)
  =e^{2\pi i\lambda c}e_{a\lambda}(x).
\]
The phase factor $e^{2\pi i\lambda c}$ does not change the modulus of any Fourier coefficient.  Thus a frame with frequencies $\Lambda$ for $\widetilde\mu$ induces a frame with frequencies $a\Lambda$ for $\mu$, with the same frame bounds.  The converse follows by replacing $T$ with $T^{-1}$.  Since $U$ is unitary and maps each exponential to a unimodular multiple of another exponential, the same argument preserves exponential orthonormal bases and exponential Riesz bases.
\end{proof}

We next introduce the symbolic coding for the Bernoulli convolution.  Let
\[
  \Omega=\{-1,1\}^{\N},
  \qquad
  Q=\left(\frac12\delta_{-1}+\frac12\delta_1\right)^{\otimes\N}.
\]
Every $\xi\in\Omega$ has the form $\xi=(\xi_1,\xi_2,\ldots)$ with $\xi_j\in\{-1,1\}$.  Define the coding map $\pi_{\rho,d}:\Omega\to\R$ by
\begin{equation}\label{eq:coding-map}
  \pi_{\rho,d}(\xi)
  =\sum_{j=1}^{\infty}\frac{d\xi_j\rho^j}{2}.
\end{equation}
We set
\(
  \mu_{\rho,d}=(\pi_{\rho,d})_\#Q,
\)
the pushforward of $Q$ under $\pi_{\rho,d}$.  Equivalently, the change-of-variable formula for pushforward measures gives
\begin{equation}\label{eq:pushforward-integral}
  \int_{\R} F(x)\,d\mu_{\rho,d}(x)
  =\int_\Omega F(\pi_{\rho,d}(\xi))\,dQ(\xi)
\end{equation}
for every bounded Borel function $F$ on $\mathbb{R}$.

\vspace{0.2cm}

The coding map \( \pi_{\rho,d}\) given in \eqref{eq:coding-map} is injective whenever $0<\rho<1/2$.
\begin{lemma}\label{lem:coding-injective}
Let $0<\rho<1/2$ and $d>0$.  For $\xi,\xi'\in\Omega$ with $\xi\ne\xi'$, we have $\pi_{\rho,d}(\xi)\ne\pi_{\rho,d}(\xi')$.  Consequently, every finite-coordinate function on $\Omega$ induces an element of $L^2(\mu_{\rho,d})$.
\end{lemma}
\begin{proof}
Let $r=\min\{j\ge1:\xi_j\ne\xi'_j\}$.  Then $|\xi_r-\xi'_r|=2$ and the $r$th term of $\pi_{\rho,d}(\xi)-\pi_{\rho,d}(\xi')$ has magnitude  $d\rho^r$, while the infinite tail starting at index $r+1$ can produce a maximum cancellation of $d\rho^{r+1}/(1-\rho)$.  We therefore obtain the separation estimate
\[
\begin{aligned}
  |\pi_{\rho,d}(\xi)-\pi_{\rho,d}(\xi')|
  \ge d\rho^r-\frac{d\rho^{r+1}}{1-\rho}
  =\frac{d\rho^r(1-2\rho)}{1-\rho}>0.
\end{aligned}
\]
This proves $\pi_{\rho,d}$ is injective.  Now let $h$ be a finite-coordinate function on $\Omega$, and define $H$ on $\pi_{\rho,d}(\Omega)$ by $H(\pi_{\rho,d}(\xi))=h(\xi)$.  Injectivity guarantees that $H$ is well defined.  Since $\Omega$ is compact and $\pi_{\rho,d}$ is continuous and injective, $\pi_{\rho,d}$ is a homeomorphism onto its image; hence $H$ is Borel measurable.  The pushforward identity \eqref{eq:pushforward-integral} gives
\begin{equation}\label{eq:l2-transfer}
  \int_{\R}|H(x)|^2\,d\mu_{\rho,d}(x)
  =\int_\Omega |h(\xi)|^2\,dQ(\xi).
\end{equation}
Thus $H\in L^2(\mu_{\rho,d})$ whenever $h\in L^2(Q)$.
\end{proof}

For a finite index set $F\subset\N$, define the Walsh packet indexed by $F$ by
\begin{equation}\label{eq:walsh-packet}
  w_F(\xi)=\prod_{j\in F}\xi_j,
  \qquad w_\varnothing=1.
\end{equation}
Each coordinate $\xi_j$ is a $\{-1,1\}$-valued Rademacher random variable.  If one works instead with binary digits $\varepsilon_j\in\{0,1\}$, then $\xi_j=2\varepsilon_j-1=-(-1)^{\varepsilon_j}$, so the two Walsh conventions differ only by the fixed factor $(-1)^{|F|}$.  Since $|w_F|=1$, one has $\|w_F\|_{L^2(Q)}=1$.

\vspace{0.2cm}

When $0<\rho<1/2$, Lemma~\ref{lem:coding-injective} permits us to transfer each Walsh packet \eqref{eq:walsh-packet} to the corresponding function $W_F\in L^2(\mu_{\rho,d})$ defined by
\begin{equation}\label{eq:walsh-transfer}
  W_F(\pi_{\rho,d}(\xi))=w_F(\xi).
\end{equation}
Its $L^2$-norm is preserved under this transform by \eqref{eq:l2-transfer}, so $\|W_F\|_{L^2(\mu_{\rho,d})}=1$.  This step relies essentially on the non-overlapping condition.  In the overlapping range, the coding map has multiple representations, and a finite-coordinate Walsh packet need not induce a well-defined function of the physical variable.

\vspace{0.2cm}

We finally derive the Fourier transform of $\mu_{\rho,d}$ from the pushforward identity.  For $\lambda\in\R$, independence of the coordinates gives
\begin{equation}\label{eq:fourier-product}
\begin{aligned}
  \widehat{\mu_{\rho,d}}(\lambda)
  =\int_\Omega e^{-2\pi i\lambda\pi_{\rho,d}(\xi)}\,dQ(\xi)
  =:\prod_{j=1}^{\infty}m_j(\lambda),
\end{aligned}
\end{equation}
where $m_j(\lambda)=\cos(\pi d\lambda\rho^j)$.
The infinite product converges, since
\[
  \sum_{j=1}^{\infty}|1-m_j(\lambda)|
  \le
  \frac{(\pi d|\lambda|)^2}{2}
  \sum_{j=1}^{\infty}\rho^{2j}<\infty.
\]
For later use, set $n_j(\lambda):=-i\sin(\pi d\lambda\rho^j)$, and define the zero set of the $j$-th cosine mask by
\begin{equation}\label{eq:pole-sets}
  P_j
  :=\{\lambda\in\R:m_j(\lambda)=0\}
  =\{\lambda\in\R:\cos(\pi d\lambda\rho^j)=0\}.
\end{equation}
Denote $P:=\bigcup_{j\ge1}P_j$ as a union of these zero sets.  The set $P_j$ is also the pole set of the tangent function at level $j$.  We therefore define
\begin{equation}\label{eq:tangent-sequence}
  u_j(\lambda):=\tan(\pi d\lambda\rho^j),
  \qquad \lambda\in\R\setminus P_j.
\end{equation}
The quotient of the sine and cosine masks is therefore
\[
\frac{n_j(\lambda)}{m_j(\lambda)}
=-iu_j(\lambda),
\qquad \lambda\notin P_j.
\]
Its precise relation to the Walsh coefficients will be established
in Lemma~\ref{lem:walsh-product}.
Thus the tangent sequence arises naturally as the quotient of the sine and cosine masks.  The exact Walsh coefficient formula and the treatment of the zero factors are given in Section~\ref{sec:walsh-frame}.

\section{Abstract decay and tangent dynamics}\label{sec:obstruction}

In this section, we collect self-contained measure-theoretic and analytic tools for the non-Fourier-frame contradiction established in Section~\ref{sec:walsh-frame}.  We establish the core $L^2(\eta)$-vanishing result for the tangent-valued sequence $(u_s)_{s\ge1}$ associated with the Bernoulli convolution masks.  Theorem~\ref{thm:abstract-lag} below formulates this abstract decay principle in full generality, with no reliance on the Bernoulli mask structure or the scaling relation $\rho^{-m}=B$. Theorem~\ref{thm:abstract-lag} plays a key role in our nonexistence argument. With this foundational result established, we focus on verifying the feasibility of conditions~\ref{ass:pointwise}--\ref{ass:product}. Proposition~\ref{prop:level-shift-dynamics} establishes the propagation implication in condition~\ref{ass:propagation} outside the pole set. In Section~\ref{sec:walsh-frame}, we show that the pole set is $\eta$-null, thereby establishing condition~\ref{ass:pointwise}, and derive condition~\ref{ass:product} from the upper frame inequality. The lower frame inequality then provides the positive lower bound used in the final contradiction.

\begin{theorem}\label{thm:abstract-lag}
Let $\eta$ be a finite positive Borel measure on $\R$, let $m\ge1$ be an integer, and let $(v_s)_{s\ge1}$ be measurable functions that are finite $\eta$-almost everywhere.  Then the decay conclusion
\[
  \lim_{s\to\infty}
  \int_{\R}|v_s(\lambda)|^2\,d\eta(\lambda)=0
\]
holds provided that the following three assumptions are satisfied.
\begin{enumerate}[label=(A\arabic*)]
\item\label{ass:pointwise} For $\eta$-almost every $\lambda\in\R$, one has $v_s(\lambda)\to0$ as $s\to\infty$.
\item\label{ass:propagation} There exist constants $R_0>0$ and $c>0$ such that, for every $s>m$ and for $\eta$-almost every $\lambda\in\R$,
\[
  |v_s(\lambda)|\ge R_0
  \quad\Longrightarrow\quad
  |v_{s-m}(\lambda)|\ge c|v_s(\lambda)|.
\]
\item\label{ass:product} There exists a constant $C<\infty$ such that, for every $s>m$,
\[
  \int_{\R}|v_{s-m}(\lambda)v_s(\lambda)|^2\,d\eta(\lambda)\le C.
\]
\end{enumerate}
\end{theorem}

\begin{proof}
Let $N_1$ denote an $\eta$-null set outside which condition~\ref{ass:pointwise} holds, and for each integer $s>m$, let $N_{2,s}$ be an $\eta$-null set outside which condition~\ref{ass:propagation} holds.  Set
\[
  N:=N_1\cup\bigcup_{s=m+1}^{\infty}N_{2,s}.
\]
Then $\eta(N)=0$, and so all subsequent pointwise arguments may be restricted to $\R\setminus N$.
Fix $R\ge R_0$ and $s>m$, and define the level set
\[
  E_{s,R}:=\{\lambda\in\R:|v_s(\lambda)|\ge R\}.
\]
For each $\lambda\in E_{s,R}\setminus N$, condition~\ref{ass:propagation} gives
\[
  |v_{s-m}(\lambda)v_s(\lambda)|^2
  \ge c^2|v_s(\lambda)|^4
  \ge c^2R^4.
\]
This estimate and condition~\ref{ass:product} imply
\[
  C
  \ge
  \int_{E_{s,R}}|v_{s-m}(\lambda)v_s(\lambda)|^2\,d\eta(\lambda)
  \ge
  c^2R^4\eta(E_{s,R}).
\]
Rearranging gives the uniform tail estimate
\begin{equation}\label{eq:tail-measure}
  \eta(E_{s,R})
  =\eta\bigl(\{\lambda\in\R:|v_s(\lambda)|\ge R\}\bigr)
  \le Cc^{-2}R^{-4}.
\end{equation}
We evaluate the truncated $L^2$-integral using the layer-cake formula (Cavalieri's principle).  For any $x\ge0$ and $R\ge0$,
\[
  x^2\mathbf 1_{\{x>R\}}
  =R^2\mathbf 1_{\{x>R\}}
   +\int_R^\infty 2t\,\mathbf 1_{\{x>t\}}\,dt.
\]
Applying this identity with $x=|v_s(\lambda)|$, Tonelli's theorem and \eqref{eq:tail-measure} give, for $R\ge R_0$,
\begin{equation}\label{eq:tail-l2}
\begin{aligned}
  \int_{\{\lambda\in\R:|v_s(\lambda)|>R\}}
    |v_s(\lambda)|^2\,d\eta(\lambda)
  &=R^2\eta\bigl(\{|v_s|>R\}\bigr)
    +\int_R^{\infty}2t\,\eta\bigl(\{|v_s|>t\}\bigr)\,dt \\
  &\le Cc^{-2}R^{-2}
    +2Cc^{-2}\int_R^{\infty}t^{-3}\,dt\\
  &=2Cc^{-2}R^{-2}.
\end{aligned}
\end{equation}
This estimate is uniform for $s>m$.

Now fix $R\ge R_0$.  Since $|v_s|^2\mathbf 1_{\{|v_s|\le R\}}\le R^2$ and $\eta$ is finite, the dominated convergence theorem and condition~\ref{ass:pointwise} give
\[
  \int_{\{\lambda\in\R:|v_s(\lambda)|\le R\}}
    |v_s(\lambda)|^2\,d\eta(\lambda)
  \longrightarrow0
  \qquad\text{as }s\to\infty.
\]
Combining this with \eqref{eq:tail-l2} gives
\[
  \limsup_{s\to\infty}
  \int_{\R}|v_s(\lambda)|^2\,d\eta(\lambda)
  \le
  \sup_{s>m}
  \int_{\{\lambda\in\R:|v_s(\lambda)|>R\}}
    |v_s(\lambda)|^2\,d\eta(\lambda)
  \le 2Cc^{-2}R^{-2}.
\]
Letting $R\to\infty$ completes the proof.
\end{proof}

We next analyze the propagation behavior of the tangent quotients arising from the Bernoulli masks.  The elementary fact is that the poles of $\tan\theta$ occur at $\theta=\pi/2\pmod\pi$.  Multiplication by an odd integer preserves this pole class, whereas multiplication by an even integer sends it to the zero class.  The lemma below resolves a minor branch-cut ambiguity introduced by the principal branch of the arctangent function when tracking the poles and finite values of $\tan(B\arctan u)$.

\begin{lemma}\label{lem:tangent-periodicity}
Let $B\ge1$ be an integer and define $T_B(u):=\tan(B\arctan u)$ whenever $\cos(B\arctan u)\ne0$, where $\arctan u\in(-\pi/2,\pi/2)$ denotes the principal branch of the arctangent function.  Fix $\theta\in\R$ with $\cos\theta\ne0$, and set $u=\tan\theta$.  Then
\[
  \cos(B\arctan u)=0
  \quad\Longleftrightarrow\quad
  \cos(B\theta)=0.
\]
In particular, if $\cos(B\theta)\ne0$, then both tangent expressions are finite and $T_B(\tan\theta)=\tan(B\theta)$.
\end{lemma}
\begin{proof}
Since $\cos\theta\ne0$, there exists an integer $k$ such that $\arctan(\tan\theta)=\theta-k\pi$.  Multiplying both sides by $B$, we obtain
\[
  B\arctan(\tan\theta)=B\theta-Bk\pi.
\]
As $Bk\in\Z$,
\[
  \cos(B\arctan(\tan\theta))
  =\cos(B\theta-Bk\pi)
  =(-1)^{Bk}\cos(B\theta).
\]
Hence the first statement follows.  If $\cos(B\theta)\ne0$, the $\pi$-periodicity of the tangent function gives the second statement.
\end{proof}

The proposition below captures the parity dichotomy central to the proof.  For odd $B$, points close to a pole of $\tan\theta$ remain close to a pole after applying $T_B$; accordingly, large values of $|u|$ produce large values of $|T_B(u)|$.  In contrast, even $B$ maps the pole class to the zero class, which forces the product $uT_B(u)$ to remain bounded.  This distinction explains why the obstruction obtained here arises only for odd integer multipliers.

\begin{proposition}\label{prop:odd-even-dynamics}
Let $B\ge2$ be an integer.
\begin{enumerate}[label=(\roman*)]
\item If $B$ is odd, then there exist constants $R_B>0$ and $c_B>0$ such that $|u|\ge R_B$ implies $|T_B(u)|\ge c_B|u|$.
\item If $B$ is even, then $uT_B(u)\to-B$ as $|u|\to\infty$.
\end{enumerate}
\end{proposition}
\begin{proof}
The function $T_B$ is odd on its natural domain, so it is enough to consider $u\to+\infty$.  We first fix the notation:
\begin{itemize}
\item $\alpha\downarrow0$ means that $\alpha>0$ and $\alpha\to0$;
\item $f\sim g$ means that $f/g\to1$ in the indicated limit.
\end{itemize}
Set $\arctan u=\pi/2-\alpha$, with $\alpha\downarrow0$.  Then $u=\cot\alpha\sim1/\alpha$.

If $B$ is odd, write $B=2\ell+1$ for some integer $\ell\ge1$.  Then
\[
  B\arctan u
  =B\left(\frac\pi2-\alpha\right)
  =\ell\pi+\frac\pi2-B\alpha,
\]
and
\[
  T_B(u)
  =\tan\left(\ell\pi+\frac\pi2-B\alpha\right)
  =\cot(B\alpha)
  \sim\frac1{B\alpha}.
\]
Recall that $u=\cot\alpha\sim1/\alpha$.  Consequently, $T_B(u)/u\to1/B$.  Thus the desired inequality in part~(i) holds for all sufficiently large $|u|$.

If $B$ is even, write $B=2\ell$.  Then
\[
  B\arctan u=\ell\pi-B\alpha,
\]
and
\[
  T_B(u)
  =\tan(\ell\pi-B\alpha)
  =-\tan(B\alpha)
  \sim-B\alpha.
\]
Since $u\sim1/\alpha$, we get $uT_B(u)\to-B$ as $|u|\to\infty$.
\end{proof}

We now transfer the properties of the single-variable tangent map $T_B$ to the sequence $(u_s)$ defined through the angular scaling $\theta_s(\lambda)=\pi d\lambda\rho^s$.  Recall from \eqref{eq:pole-sets} that $P_j$ is the zero set of the $j$-th cosine mask and hence the pole set of $u_j$.  The algebraic identity $\rho^{-m}=B$ yields the exact identity $\theta_{s-m}=B\theta_s$ for every $\lambda$.  As a result, the large-magnitude propagation property obtained above holds uniformly for all $s>m$ and all $\lambda$ outside $P_s\cup P_{s-m}$.

\begin{proposition}\label{prop:level-shift-dynamics}
Assume $0<\rho<1$, $d>0$, and let integers $m\ge1$ and $B\ge2$ satisfy the algebraic identity $\rho^{-m}=B$.  Then, for $s>m$ and $\lambda\notin P_s\cup P_{s-m}$, the quantities $u_{s-m}(\lambda)$ and $T_B(u_s(\lambda))$ are finite and
\[
  u_{s-m}(\lambda)=T_B(u_s(\lambda)).
\]
When $B$ is odd, there exist constants $R_B>0$ and $c_B>0$, independent of $s$ and $\lambda$, such that, for every $s>m$ and every $\lambda\notin P_s\cup P_{s-m}$,
\[
  |u_s(\lambda)|\ge R_B
  \quad\Longrightarrow\quad
  |u_{s-m}(\lambda)|\ge c_B|u_s(\lambda)|.
\]
\end{proposition}
\begin{proof}
Let $\theta_s(\lambda):=\pi d\lambda\rho^s$.  The identity $\rho^{-m}=B$ implies
\[
  \theta_{s-m}(\lambda)
  =\pi d\lambda\rho^{s-m}
  =B\theta_s(\lambda).
\]
For $\lambda\notin P_s\cup P_{s-m}$, we have $\cos\theta_s(\lambda)\ne0$ and $\cos(B\theta_s(\lambda))=\cos\theta_{s-m}(\lambda)\ne0$.  Lemma~\ref{lem:tangent-periodicity} gives
\[
  u_{s-m}(\lambda)
  =\tan(B\theta_s(\lambda))
  =T_B(\tan\theta_s(\lambda))
  =T_B(u_s(\lambda)).
\]
For odd $B$, the large-magnitude comparison with constants $R_B$ and $c_B$ follows immediately from Proposition~\ref{prop:odd-even-dynamics}.
\end{proof}

With $P:=\bigcup_{j\ge1}P_j$, Proposition~\ref{prop:level-shift-dynamics} verifies condition~\ref{ass:propagation} on $\R\setminus P$.  Moreover, since $\rho^s\to0$, one has $u_s(\lambda)\to0$ as $s\to\infty$ for every $\lambda\in\R\setminus P$.  In Section~\ref{sec:walsh-frame}, we prove that $P$ is null with respect to the frame-induced measure $\eta$, thereby completing the verification of condition~\ref{ass:pointwise}, and derive condition~\ref{ass:product} from the upper frame inequality.

\section{Walsh quotients and the frame contradiction}\label{sec:walsh-frame}

This section turns the mask quotient described in Section~\ref{sec:preliminaries} into exact Fourier coefficient identities. These identities allow us to apply the frame inequalities to two-coordinate and one-coordinate Walsh packets. The former yields condition~\ref{ass:product}, whereas the latter gives a uniform positive lower bound for the $L^2(\eta)$-mass of $u_s$. Together with Theorem~\ref{thm:abstract-lag}, these estimates yield the desired contradiction.

Throughout this section, we assume $0<\rho<1/2$ and $d>0$, so that the functions $W_F$ in \eqref{eq:walsh-transfer} are well defined in $L^2(\mu_{\rho,d})$.  Recall the masks $m_j,n_j$ defined in Section~\ref{sec:preliminaries} and the tangent variables $u_j$ in \eqref{eq:tangent-sequence}.  Their exponential forms are
\begin{equation}\label{eq:mjnj-general}
\begin{aligned}
  m_j(\lambda)
  &=\frac12\left(e^{-\pi i d\lambda\rho^j}
      +e^{\pi i d\lambda\rho^j}\right),\\
  n_j(\lambda)
  &=\frac12\left(e^{-\pi i d\lambda\rho^j}
      -e^{\pi i d\lambda\rho^j}\right).
\end{aligned}
\end{equation}
Whenever $m_j(\lambda)\ne0$, define
\begin{equation}\label{eq:qj-definition}
  q_j(\lambda)
  :=\frac{n_j(\lambda)}{m_j(\lambda)}
  =-iu_j(\lambda).
\end{equation}

We shall use expectation notation associated with the product probability space \((\Omega,Q)\): for $X\in L^1(\Omega,Q)$, define $\E X=\int_\Omega X(\xi)\,dQ(\xi)$.  As $Q$ is a product measure and the coordinate random variables $\{\xi_j\}$ are independent, the expectation of a finite product of single-coordinate functions $G_j(\xi_j)$ satisfies the factorization identity \[\E\prod_jG_j(\xi_j)=\prod_j\E G_j(\xi_j).\]  This elementary factorization is the only probabilistic fact we rely on for the subsequent lemma.

\begin{lemma}\label{lem:walsh-product}
Let $F\subset\N$ be finite.  Then
\begin{equation}\label{eq:walsh-product}
  \langle W_F,e_\lambda\rangle_{L^2(\mu_{\rho,d})}
  =\prod_{j\in F}n_j(\lambda)
   \prod_{j\notin F}m_j(\lambda).
\end{equation}
In particular, if $m_j(\lambda)\ne0$ for every $j\in F$, then
\begin{equation}\label{eq:walsh-quotient-formula}
  \langle W_F,e_\lambda\rangle_{L^2(\mu_{\rho,d})}
  =\widehat{\mu_{\rho,d}}(\lambda)
   \prod_{j\in F}q_j(\lambda).
\end{equation}
\end{lemma}
\begin{proof}
Set $a_j=\pi d\lambda\rho^j$.  By the definition of the inner product and the symbolic representation of $\mu_{\rho,d}$,
\begin{align*}
  \langle W_F,e_\lambda\rangle_{L^2(\mu_{\rho,d})}
  &=\int_\Omega w_F(\xi)e^{-2\pi i\lambda\pi_{\rho,d}(\xi)}\,dQ(\xi)\\
  &=\E\left[
    \prod_{j\in F}\xi_j
    \prod_{j=1}^{\infty}e^{-ia_j\xi_j}
    \right].
\end{align*}
Choose $N$ sufficiently large such that $F\subset\{1,\ldots,N\}$, and define the truncated expectation
\[
  I_N(\lambda)
  :=\E\left[
    \prod_{j\in F}\xi_j
    \prod_{j=1}^{N}e^{-ia_j\xi_j}
    \right].
\]
By independence of the coordinates, the finite product factors as
\begin{equation}\label{eq:split-truncated-ex}
  I_N(\lambda)
  =\prod_{j\in F}\E\left[\xi_j e^{-ia_j\xi_j}\right]
   \prod_{\substack{1\le j\le N\\ j\notin F}}
    \E\left[e^{-ia_j\xi_j}\right].
\end{equation}
Since each $\xi_j$ takes the values $1$ and $-1$ with probability $1/2$, the exponential forms in \eqref{eq:mjnj-general} give
\[
  \E\left[e^{-ia_j\xi_j}\right]
  =\frac12(e^{-ia_j}+e^{ia_j})=m_j(\lambda),
\]
and
\[
  \E\left[\xi_j e^{-ia_j\xi_j}\right]
  =\frac12(e^{-ia_j}-e^{ia_j})=n_j(\lambda).
\]
Substituting these identities into \eqref{eq:split-truncated-ex} gives
\[
  I_N(\lambda)
  =\prod_{j\in F}n_j(\lambda)
   \prod_{\substack{1\le j\le N\\ j\notin F}}m_j(\lambda).
\]
The truncated exponential products converge pointwise to the full infinite product, and their moduli are bounded by $1$.  The dominated convergence theorem therefore gives
\[
  I_N(\lambda)
  \longrightarrow
  \langle W_F,e_\lambda\rangle_{L^2(\mu_{\rho,d})}
  \qquad\text{as }N\to\infty.
\]
On the other hand, the infinite cosine product converges by the standard infinite-product criterion, since
\[
  \sum_{j=1}^{\infty}|1-m_j(\lambda)|
  \le\frac12\sum_{j=1}^{\infty}(\pi d|\lambda|\rho^j)^2<\infty.
\]
Hence
\[
  \prod_{j\in F}n_j(\lambda)
   \prod_{\substack{1\le j\le N\\ j\notin F}}m_j(\lambda)
  \longrightarrow
  \prod_{j\in F}n_j(\lambda)
   \prod_{j\notin F}m_j(\lambda).
\]
This proves \eqref{eq:walsh-product}.  If $m_j(\lambda)\ne0$ for every $j\in F$, then
\begin{align*}
  \prod_{j\in F}n_j(\lambda)
   \prod_{j\notin F}m_j(\lambda)
  &=\left(\prod_{j=1}^{\infty}m_j(\lambda)\right)
    \prod_{j\in F}\frac{n_j(\lambda)}{m_j(\lambda)}\\
  &=\widehat{\mu_{\rho,d}}(\lambda)
    \prod_{j\in F}q_j(\lambda),
\end{align*}
which gives \eqref{eq:walsh-quotient-formula} by \eqref{eq:qj-definition}.
\end{proof}

We now address the zero factors suppressed by the quotient notation. The quotient formula from Lemma~\ref{lem:walsh-product} is immediate at frequencies where $\widehat{\mu_{\rho,d}}(\lambda)\ne0$. The central subtlety concerns the zero set of   $\widehat{\mu_{\rho,d}}$.

\begin{lemma}\label{lem:zero-completed-single}
Suppose $\rho^{-m}=B$, where $m\ge1$ and $B\ge3$ is odd.  Then, for $s>m$ and every $\lambda\in\R$,
\[
  \langle W_{\{s\}},e_\lambda\rangle_{L^2(\mu_{\rho,d})}
  =
  \begin{cases}
    \widehat{\mu_{\rho,d}}(\lambda)q_s(\lambda),
    & \widehat{\mu_{\rho,d}}(\lambda)\ne0,\\[3pt]
    0,
    & \widehat{\mu_{\rho,d}}(\lambda)=0.
  \end{cases}
\]
\end{lemma}
\begin{proof}
	If \(\widehat{\mu_{\rho,d}}(\lambda)\ne0\),  its infinite product representation \(\widehat{\mu_{\rho,d}}(\lambda)=\prod_{j=1}^{\infty}m_j(\lambda)\) implies that \(m_j(\lambda)\ne 0\) for all \(j\ge 1\). Applying  Lemma~\ref{lem:walsh-product} to the singleton index set \(F=\{s\}\)  gives that
\[\langle W_{\{s\}},e_\lambda\rangle=\widehat{\mu_{\rho,d}}(\lambda)q_s(\lambda).\]			
If \(\widehat{\mu_{\rho,d}}(\lambda)=0\),  by the product formula \eqref{eq:fourier-product}, this means that \(m_r(\lambda)=0\) for some \(r\ge1\).  Otherwise, if \(m_j(\lambda)\ne 0\) for all \(j\ge 1\), then \(\sum_{j=1}^{\infty}|1-m_j(\lambda)|<\infty\) would imply the infinite product \(\widehat{\mu_{\rho,d}}(\lambda)\) converges to a nonzero complex number, since all but finitely many factors \(m_j(\lambda)=\cos(\pi d\lambda\rho^{j})\) lie within a small positive neighborhood of \(1\) and remain strictly positive.			
By Lemma~\ref{lem:walsh-product},
\[
\langle W_{\{s\}},e_\lambda\rangle
=n_s(\lambda)\prod_{j\ne s}m_j(\lambda).
\]
If \(r\ne s\), then the zero factor \(m_r(\lambda)\) remains in the product \(\prod_{j\ne s}m_j(\lambda)\), so \(\langle W_{\{s\}},e_\lambda\rangle=0\).  If \(r=s\), then \(\cos(\pi d\lambda\rho^s)=0\), so \(\pi d\lambda\rho^s=(2k+1)\pi/2\) for some \(k\in\Z\).  Multiplying both sides by \(\rho^{-m}=B\) gives
\[
\pi d\lambda\rho^{s-m}=B\pi d\lambda\rho^s=\frac{B(2k+1)\pi}{2}.
\]
Since \(B\) is odd, \(B(2k+1)\) is odd. Therefore \(m_{s-m}(\lambda)=\cos(\pi d\lambda\rho^{s-m})=0\).  As \(s-m\ne s\), this zero factor still occurs in \(\prod_{j\ne s}m_j(\lambda)\).  Hence \(\langle W_{\{s\}},e_\lambda\rangle=0\) in the zero case as well.
\end{proof}

\begin{remark}
Set $u_s=0$ on $\{\lambda\in\R:\widehat{\mu_{\rho,d}}(\lambda)=0\}$.  Under this convention, Lemma~\ref{lem:zero-completed-single} gives
\[
  |\langle W_{\{s\}},e_\lambda\rangle_{L^2(\mu_{\rho,d})}|^2
  =|\widehat{\mu_{\rho,d}}(\lambda)|^2|u_s(\lambda)|^2
\]
for every $\lambda\in\R$.
\end{remark}

We now assume that a Fourier frame exists on $L^2(\mu_{\rho,d})$.  Let $\Lambda\subset\R$ be a Fourier frame spectrum for $\mu_{\rho,d}$, with lower and upper frame bounds $A$ and $C$, respectively, and let $\eta$ be the frame-induced measure defined in \eqref{eq:eta-def}.  As noted in the Introduction, $\eta$ is finite.

Recall from \eqref{eq:pole-sets} that $P_j=\{\lambda:m_j(\lambda)=0\}$, and set $P=\bigcup_{j=1}^{\infty}P_j$.  If $\lambda\in P_j$, then the product representation of $\widehat{\mu_{\rho,d}}(\lambda)$ contains a zero factor, so $\widehat{\mu_{\rho,d}}(\lambda)=0$.  It follows that $\eta(P_j)=0$ for every $j$, and countable subadditivity gives $\eta(P)=0$.  We use the representative convention from the preceding remark, so $u_j=0$ on $P$ for every $j$; this convention does not change any integral with respect to $\eta$.

For $\lambda\notin P$, all $u_s(\lambda)=\tan(\pi d\lambda\rho^s)$ are finite and $u_s(\lambda)\to0$ as $s\to \infty$ since $\rho^s\to0$.  On the null set $P$, the convergence still holds by our representative convention $u_s=0$ on  $P$.   Condition~\ref{ass:pointwise} of Theorem~\ref{thm:abstract-lag} is therefore satisfied.

The following proposition serves as a key bridge from frame theory to the abstract assumptions of Theorem~\ref{thm:abstract-lag}.  It uses the two frame inequalities for different purposes.  The upper frame bound, applied to a two-coordinate Walsh packet, yields the required product estimate.  The lower frame bound, tested on a single-coordinate Walsh packet, gives a uniform positive lower bound.

\begin{proposition}
Suppose $0<\rho<1/2$, $d>0$, and $\rho^{-m}=B$, where $m\ge1$ is an integer and $B\ge3$ is odd.  Let $\Lambda$ be a Fourier frame spectrum for $\mu_{\rho,d}$ with frame bounds $A$ and $C$, and let $\eta$ be given by \eqref{eq:eta-def}.  Then, for every $s>m$,
\begin{equation}\label{eq:lag-product-bound}
  \int_{\R}|u_{s-m}(\lambda)u_s(\lambda)|^2\,d\eta(\lambda)\le C,
\end{equation}
and
\begin{equation}\label{eq:single-lower-bound}
  \int_{\R}|u_s(\lambda)|^2\,d\eta(\lambda)\ge A.
\end{equation}
\end{proposition}
\begin{proof}
We begin by applying the upper frame inequality to the two-coordinate Walsh packet $W_{\{s-m,s\}}$.  Since $\|W_{\{s-m,s\}}\|_{L^2(\mu_{\rho,d})}=1$, Lemma~\ref{lem:walsh-product}, together with $\eta(P)=0$, gives
\[
\begin{aligned}
  \int_{\R}|u_{s-m}(\lambda)u_s(\lambda)|^2\,d\eta(\lambda)
  &=\sum_{\lambda\in\Lambda\setminus P}
    |\langle W_{\{s-m,s\}},e_\lambda\rangle|^2\\
  &\le\sum_{\lambda\in\Lambda}
    |\langle W_{\{s-m,s\}},e_\lambda\rangle|^2
  \le C.
\end{aligned}
\]
This proves \eqref{eq:lag-product-bound}.

For the single-coordinate Walsh packet $W_{\{s\}}$, Lemma~\ref{lem:zero-completed-single}, the preceding remark, and the lower frame inequality give
\[
\begin{aligned}
  A
  &\le\sum_{\lambda\in\Lambda}
    |\langle W_{\{s\}},e_\lambda\rangle|^2\\
  &=\sum_{\lambda\in\Lambda}
    |\widehat{\mu_{\rho,d}}(\lambda)|^2|u_s(\lambda)|^2
  =\int_{\R}|u_s(\lambda)|^2\,d\eta(\lambda).
\end{aligned}
\]
This proves \eqref{eq:single-lower-bound}.
\end{proof}

\begin{proof}[\bf Proof of Theorem~\ref{thm:level-shift-main}]
Suppose, to the contrary, that $\Lambda$ is a Fourier frame spectrum for $\mu_{\rho,d}$ with lower and upper frame bounds $A$ and $C$, respectively, and let $\eta$ be the frame-induced measure defined in \eqref{eq:eta-def}. Condition~\ref{ass:pointwise} has already been established. Since $P_s\cup P_{s-m}\subset P$ and $\eta(P)=0$, Proposition~\ref{prop:level-shift-dynamics} gives condition~\ref{ass:propagation}. Moreover, \eqref{eq:lag-product-bound} gives condition~\ref{ass:product}. Hence Theorem~\ref{thm:abstract-lag} yields
\[
  \int_{\R}|u_s(\lambda)|^2\,d\eta(\lambda)
  \longrightarrow0
  \qquad\text{as }s\to\infty.
\]
On the other hand, \eqref{eq:single-lower-bound} gives
\[
  \int_{\R}|u_s(\lambda)|^2\,d\eta(\lambda)\ge A,
  \qquad s>m,
\]
which is a contradiction. Therefore $L^2(\mu_{\rho,d})$ admits no Fourier frame.
\end{proof}

\section{Scope and the remaining problem}\label{sec:scope}

The parameter range in Theorem~\ref{thm:level-shift-main} is $\rho=B^{-1/m}$, where $B\ge3$ is odd and $B>2^m$. The last inequality is equivalent to $0<\rho<1/2$.  The case $m=1$ gives the odd reciprocal-integer scales. The case $m>1$ includes a family of Bernoulli convolutions with non-integer contraction ratios that admit no Fourier frames. This motivates further investigation into the existence of Fourier frames for Bernoulli convolutions in the remaining parameter regimes, analogous to existing work on spectrality for such measures.

The assumption $0<\rho<1/2$ is essential to guarantee the symbolic coding map is injective as in Lemma~\ref{lem:coding-injective}, so that Walsh packets on $\Omega$ are well-defined functions in $L^2(\mu_{\rho,d})$.  For $\rho>1/2$, the coding has overlaps, and the finite-coordinate Walsh packets would have to be replaced by their conditional expectations with respect to the full random series.  This breaks the clean product factorization of Fourier coefficients central to our argument.

The condition $\rho^{-m}=B$ with odd $B$  is essential for the present method.  It supplies the exact relation $\theta_{s-m}=B\theta_s$ , which is used  for large-magnitude propagation of tangent quotients, and for propagation of zero factors.  For a general contraction parameter $\rho$, such an exact identity need not exist.  Diophantine near-identities may occur, but they do not give the uniform pole propagation required by Theorem~\ref{thm:abstract-lag}.  Thus the theorem proves an obstruction in the odd reciprocal-power regime, while the remaining non-overlapping parameters require a different mechanism.

This nonexistence result suggests a natural extension of Dai's full spectrality classification for Bernoulli convolutions to the broader setting of  Fourier frames over non-overlapping self-similar measures.
 This is a conjecture rather than a consequence of the present method.
 
\begin{conjecture}
Let $0<\rho\le1/2$ and $d>0$.  Then $L^2(\mu_{\rho,d})$ admits a Fourier frame  if and only if $\rho=1/(2k)$ for some $k\in\N$.  Equivalently, in the non-overlapping range, the existence of an exponential Fourier frame should coincide with spectrality for Bernoulli convolutions.
\end{conjecture}

The obstruction arises from the interaction between the recursive mask factorization of the Bernoulli convolution Fourier transform and the upper and lower frame inequalities.

\begingroup
\setlength{\parskip}{0pt}

\bigskip

\noindent\textsc{Xiao-Ye Fu}\\
Hubei Key Laboratory of Mathematical Sciences, College of Mathematics and Statistics,
Central China Normal University, Wuhan, Hubei 430079, China\\
Email address: \texttt{xiaoyefu@ccnu.edu.cn}

\medskip

\noindent\textsc{Zi-Jian Song}\\
Hubei Key Laboratory of Mathematical Sciences, College of Mathematics and Statistics,
Central China Normal University, Wuhan, Hubei 430079, China\\
Email address: \texttt{songzijian2025@outlook.com}

\medskip

\noindent\textsc{Wei-Jie Wang}\\
Hubei Key Laboratory of Mathematical Sciences, College of Mathematics and Statistics,
Central China Normal University, Wuhan, Hubei 430079, China\\
Email address: \texttt{wwjmath@163.com}

\endgroup

\end{document}